\documentclass[11pt,twoside]{article}
\usepackage{amsmath,amsfonts,amscd,amsthm}

\newtheorem{thm}{Theorem}[section]
\newtheorem{thm0}{Theorem}
\newtheorem{prop}[thm]{Proposition}
\newtheorem{lem}[thm]{Lemma}
\newtheorem{cor}[thm]{Corollary}
\newtheorem{cor0}[thm0]{Corollary}
\theoremstyle{definition}
\newtheorem{defin}[thm]{Definition}

\theoremstyle{remark}

\numberwithin{equation}{section}

\providecommand\ufootnote[1]{{\let\thefootnote\relax\footnote[0]{#1}}}

\newcommand{\ac}{\mathcal A}
\newcommand{\bc}{\mathcal B}

\newcommand{\cc}{\mathcal C}
\newcommand{\dc}{\mathcal D}
\newcommand{\ec}{\mathcal E}
\newcommand{\hc}{\mathcal H}
\newcommand{\lc}{\mathcal L}

\newcommand{\rb}{\mathbb R}
\newcommand{\nb}{\mathbb N}
\newcommand{\mb}{\mathbb M}
\newcommand{\cb}{\mathbb C}

\newcommand{\ci}{{\mathcal C}^\infty}

\newcommand{\ol}{\overline}

\newcommand{\pa}{\partial}
\newcommand{\opa}{\ol\partial}
\newcommand{\wt}{\widetilde}

\begin{document}

\title{Stability of the embeddability under perturbations of the CR structure for compact CR manifolds}

\author{Christine LAURENT-THI\'{E}BAUT}

\date{}

\maketitle

\begin{abstract}
We study the stability of the embeddability
of compact $2$-concave CR manifolds in complex manifolds under small
horizontal perturbations of the CR structure.
\end{abstract}

\ufootnote{\hskip-0.6cm UJF-Grenoble 1, Institut Fourier, Grenoble,
F-38041, France\newline CNRS UMR 5582, Institut Fourier,
Saint-Martin d'H\`eres, F-38402, France}

\ufootnote{\hskip-0.6cm  {\it A.M.S. Classification}~: 32V30, 32V05,
32V20.\newline {\it Key words}~: CR structures, embeddings, homotopy
formula.}

\bibliographystyle{amsplain}

The study of local and global embeddability of CR manifolds in
complex manifolds has occupied a large number of mathematicians in
the last forty years. Most of the results concern the case of
strictly pseudoconvex CR manifolds of hypersurface type, very few is
known in the other cases.

The stability of the embeddability property of a CR manifold $M$ was
first studied by N. Tanaka, \cite{Ta}, for strictly pseudoconvex CR
manifolds of hypersurface type and real dimension greater or equal to
$5$ embedded in some $\cb^N$.
Few years later R. S. Hamilton, \cite{Ha1} et \cite{Ha2}, was
interested in the
stability of the embeddability property for
hypersurfaces provided the perturbation of the original CR structure
is the restriction of a perturbation of a complex structure on some
complex manifold $X$ of which $M$ is the boundary. It was proved in
\cite{Ki1} and \cite{Ki2} that this last condition is satisfied for
$2$-concave hypersurfaces with a perturbation preserving the contact
structure.

Here we consider CR manifolds of higher codimension type
with mixed Levi signature and we are interested in the stability of
the embeddability of such CR manifolds in complex manifolds under
small perturbations of the CR structure preserving the complex
tangent bundle.

Our main result is the following theorem:

\begin{thm0}\label{perturb}
Let $(\mb,H_{0,1}\mb)$ be an abstract compact CR
manifold of class $\cc^\infty$, which is  smoothly embeddable as a
CR manifold in a complex manifold $X$. Assume that $(\mb,H_{0,1}\mb)$
is $2$-concave. Let $\ec_0$ be a smooth
embedding  from $\mb$ into $X$ and denote by $M_0$ the image of $\mb$
by the embedding $\ec_0$. Assume that either the $\opa_b$-group of
cohomology $H^{0,1}(M_0,T_{1,0}X_{|_{M_0}})=0$ or $X=\cb P^n$, and let $\widehat
H_{0,1}\mb$ be an horizontal perturbation of $H_{0,1}\mb$ defined by
a $(0,1)$-form $\Phi\in\cc_{0,1}^{l+3}(\mb,H_{1,0}\mb)$, $l\geq 1$.

Then there exists a positive real number $\delta$ such that if
$\|\Phi\|_{l+3}<\delta$, then the CR manifold $(\mb,\widehat
H_{0,1}\mb)$ is embeddable in $X$ as a CR submanifold of class
$\cc^l$.
\end{thm0}

Looking to the proof (see section \ref{s4}), one can easily see that $\cb P^n$ can be replaced by $\cb^N\times\cb P^{n_1}\times\dots\times\cb P^{n_k}$ in Theorem \ref{perturb}
\medskip

In the case of a general complex manifold $X$, Theorem \ref{perturb} was proved
by P. L. Polyakov in \cite{Poly2}
under the stronger hypothesis that $\mb$ is $3$-concave,
and $\ec_0$ is a generic embedding,
but with a larger loss of regularity. We have to notice that, compare to
our situation, Polyakov has no restriction on the kind of the
perturbation of the CR structure.
For example let us consider the CR manifold $M=\{(z,\zeta)\in\cb P^n\times\cb
P^n~|~\sum_{j=1}^{n+1} z_j\ol\zeta_j=0\}$ which is $2$-concave if $n\geq 3$, then
by Theorem \ref{perturb} any sufficiently small horizontal CR perturbation of $\mb$ is still embeddable in $\cb P^n\times\cb
P^n$, but,  if $n\geq 4$, $\mb$ satisfies
the vanishing cohomological condition  (see \cite{Poly2}),
and is $3$-concave, so by Polyakov's result any sufficiently small CR perturbation (without any restriction) of $\mb$ is still embeddable in $\cb P^n\times\cb
P^n$.

Note that our hypothesis
of horizontality of
the perturbation allows us to work with anisotropic H\"older spaces
and to avoid a Nash-Moser process in the proof of the theorem.
\medskip

Finally note that in the  case of a Levi non degenerate hypersurface, by a theorem of Gray \cite{Gr}
saying that all contact structures on a compact manifold near a fixed
contact structure are equivalent, any perturbation of the CR structure can be reduced to an
horizontal one (in that case the horizontal perturbations are the perturbations which preserve the natural contact structure). So we can immediately derive from Theorem \ref{perturb} the following sharp corollary

\begin{cor0}
Let $\mb$ be a $2$-concave non degenerate real
quadric of $\cb P^n$ defined by
$$M=\{z\in\cb P^n~|~
\sum_{j=0}^p|z_j|^2-\sum_{j=1}^q|z_{p+j}|^2=0\},$$
with $p\geq 2$, $q\geq 3$ and $p+q=n$. If we equip $\mb$ with a sufficiently small perturbation of class $\cc^{l+3}$ of the CR structure induced by the complex structure of $\cb P^n$, then the new CR manifold is still embeddable in $\cb P^n$ as a CR submanifold of class $\cc^l$.
\end{cor0}

The case of the $1$-concave non degenerate quadrics was studied by Biquard in \cite{Bi}. He proved that the space of obstruction to the stability of the embeddability is infinite dimensional.
\medskip

The paper is organized as follows:

Sections \ref{s1} and \ref{s2} consist in the description of the general setting
and in the definitions of the the main objects used in this paper.

Section \ref{s3} is devoted to the proof of Theorem \ref{perturb} in the general case of a complex manifold.
We first remark that the problem can be reduced to the solvability
of some tangential Cauchy-Riemann equation for the perturbed
structure. Using global homotopy formulas with good estimates we are
lead to a fixed point theorem, which gives the solution.

In section \ref{s4}, we consider the case when $X=\cb P^n$.

\section{CR Structures}\label{s1}

Let $\mb$ be a $\cc^l$-smooth, $l\geq 1$, paracompact differential manifold, we
denote by $T\mb$ the tangent bundle of $\mb$ and by $T_\cb
\mb=\cb\otimes T\mb$ the complexified tangent bundle.

\begin{defin}\label{structure}
An \emph {almost CR structure} on $\mb$ is a subbundle $H_{0,1}\mb$ of
$T_\cb \mb$ such that $H_{0,1}\mb\cap \ol{H_{0,1}\mb}=\{0\}$.

If the almost CR structure is integrable, i.e. for all
$Z,W\in\Gamma(\mb,H_{0,1}\mb)$ then $[Z,W]\in\Gamma(\mb,H_{0,1}\mb)$ , then it is called a \emph {CR structure}.

If $H_{0,1}\mb$ is a CR structure, the pair $(\mb,H_{0,1}\mb)$ is
called an \emph {abstract CR manifold}.

The \emph{CR dimension} of $\mb$ is defined by CR-dim
$\mb=\rm{rk}_\cb~H_{0,1}\mb$.
\end{defin}

We set $H_{1,0}\mb=\ol{H_{0,1}\mb}$ and $H\mb=H_{1,0}\mb\oplus H_{0,1}\mb$.

\begin{defin}\label{morphisme}
Let $(\mb,H_{0,1}\mb)$ and $(\mb',H_{0,1}\mb')$ be two CR manifolds and $F~:~\mb\to\mb'$ a $\cc^1$-map. The map $F$ is called a \emph{CR map} if and only if for each $x\in\mb$, $dF((H_{0,1}\mb)_x)\subset (H_{0,1}\mb')_{F(x)}$.
\end{defin}

In particular, if $(\mb,H_{0,1}\mb)$ is a CR manifold and $f$ a complex valued function, then $f$ is a CR function if and only if for any $\ol L\in H_{0,1}\mb$ we have $\ol L f=0$.
\medskip

We denote by $H^{0,1}\mb$ the dual bundle $(H_{0,1}\mb)^*$ of $H_{0,1}\mb$.
Let $\Lambda^{0,q}\mb=\bigwedge^q(H^{0,1}\mb)$, then
$\cc^s_{0,q}(\mb)=\Gamma^s(\mb,\Lambda^{0,q}\mb)$ is called the
space of $(0,q)$-forms of class $\cc^s$, $0\leq s\leq l$ on $\mb$.

If the almost CR structure is a CR structure, i.e. if it is
integrable, and if $s\geq 1$, then we can define an operator
\begin{equation}
\opa_b~:~\cc^s_{0,q}(\mb)\to \cc^{s-1}_{0,q+1}(\mb)
\end{equation}
called the \emph{tangential Cauchy-Riemann operator} by setting
$\opa_b f=df_{|_{H_{0,1}\mb\times\dots\times H_{0,1}\mb}}$. It
satisfies $\opa_b\circ\opa_b=0$ and a complex valued function $f$ is a CR function if and only if $\opa_b f=0$.

For $X\in H\mb$, let $X^{0,1}$ denote the projection of $X$ on $H_{0,1}\mb$ and $X^{1,0}$ its projection on $H_{1,0}\mb$.

For $(0,q)$-forms with values in the vector bundle $H_{1,0}\mb$, we define the tangential Cauchy-Riemann operator by setting
$$\opa_b u(\ol L)=[\ol L, u]^{1,0},$$
if $u\in\Gamma(\mb,H_{1,0}\mb)$ and $\ol L\in\Gamma(\mb,H_{0,1}\mb)$ and
\begin{align*}
\opa_b\theta(\ol L_1,\dots,\ol L_{q+1})&=\sum_{j=1}^{q+1}(-1)^{j+1}\opa_b(\theta(\ol L_1,\dots,\widehat{\ol L_j},\dots,\ol L_{q+1}))(\ol L_j)\\&+\sum_{j<k}(-1)^{j+k}\theta([\ol L_j,\ol L_k],\ol L_1,\dots,\widehat{\ol L_j},\dots,\widehat{\ol L_k},\dots,\ol L_{q+1}),
\end{align*}
if $\theta\in \ci_{0,q}(\mb,H_{1,0}\mb)$ and $\ol L_1,\dots,\ol L_{q+1}\in \Gamma(\mb,H_{0,1}\mb)$. In particular for $q=1$, we get
$$\opa_b\theta(\ol L_1,\ol L_2)=[\ol L_1,\theta(\ol L_2)]^{1,0}-[\ol L_2,\theta(\ol L_1)]^{1,0}-\theta([\ol L_1,\ol L_2]).$$

\medskip

The annihilator $H^0\mb$ of $H\mb=H_{1,0}\mb\oplus H_{0,1}\mb$ in
$T_\cb^*\mb$ is called the \emph{characteristic bundle} of $\mb$.
Given $p\in\mb$, $\omega\in H^0_p\mb$ and $X,Y\in H_p\mb$, we choose
$\wt\omega\in\Gamma(\mb,H^0\mb)$ and $\wt X,\wt Y\in
\Gamma(\mb,H\mb)$ with $\wt\omega_p=\omega$, $\wt X_p=X$ and $\wt
Y_p=Y$. Then $d\wt\omega(X,Y)=-\omega([\wt X,\wt Y])$. Therefore we can
associate to each $\omega\in H^0_p\mb$ an hermitian form
\begin{equation}\label{levi}
L_\omega(X)=-i\omega([\wt X,\ol{\wt X}])
\end{equation}
on $H_p\mb$. This is called the \emph{Levi form} of $\mb$ at
$\omega\in H^0_p\mb$.

In the study of the $\opa_b$-complex two important geometric
conditions were introduced for CR manifolds of real dimension $2n-k$
and CR-dimension $n-k$. The first one by Kohn in the hypersurface
case, $k=1$, the condition Y(q), the second one by Henkin in
codimension $k$, $k\ge 1$, the $q$-concavity.

A CR manifold $\mb$ satisfies Kohn's condition $Y(q)$ at a point
$p\in\mb$ for some $0\le q\le n-1$, if the Levi form of $\mb$ at $p$
has at least $\max(n-q,q+1)$ eigenvalues of the same sign or at
least $\min(n-q,q+1)$ eigenvalues of opposite signs.

A CR manifold $\mb$ is said to be \emph{$q$-concave} at $p\in\mb$
for some $0\le q\le n-k$, if the Levi form $L_\omega$ at $\omega\in
H^0_p\mb$ has at least $q$ negative eigenvalues on $H_p\mb$ for
every nonzero $\omega\in H^0_p\mb$.

In \cite{ShWa} the condition Y(q) is extended to arbitrary
codimension.
\begin{defin}\label{y(q)}
An abstract CR manifold is said to satisfy \emph{condition Y(q)} for
some $1\le q\le n-k$ at $p\in\mb$ if the Levi form $L_\omega$ at
$\omega\in H^0_p\mb$ has at least $n-k-q+1$ positive eigenvalues or
at least $q+1$ negative eigenvalues on $H_p\mb$ for every nonzero
$\omega\in H^0_p\mb$.
\end{defin}

Note that in the hypersurface type case, i.e. $k=1$, this condition is
equivalent to the classical condition Y(q) of Kohn for
hypersurfaces and in particular if the CR structure is strictly peudoconvex,
i.e. the Levi form is positive definite or negative definite,
condition Y(q) holds for all $1\leq q<n-1$. Moreover, if $\mb$ is
$q$-concave at $p\in\mb$, then
$q\le (n-k)/2$ and condition Y(r) is satisfied at $p\in\mb$ for any
$0\le r\le q-1$ and $n-k-q+1\le r\le n-k$.

\begin{defin}\label{plongement}
Let $(\mb, H_{0,1}\mb)$ be an abstract CR manifold, $X$ be  a complex
manifold  and $F~:~\mb\to X$ be an embedding of class $\cc^l$, then
$F$ is called a \emph{CR embedding} if $dF(H_{0,1}\mb)$ is a subbundle
of the bundle $T_{0,1}X$ of the holomorphic vertor fields of $X$ and
$dF(H_{0,1}\mb)=T_{0,1}X\cap T_\cb F(\mb)$.
\end{defin}

Let $F$ be a CR embedding of an abstract CR manifold into a complex
manifold $X$ and set $M=F(\mb)$, then $M$ is a CR manifold with the
CR structure $H_{0,1}M=T_{0,1}X\cap T_\cb M$.

Let $U$ be a coordinate domain in $X$, then
$F_{|_{F^{-1}(U)}}=(f_1,\dots ,f_N)$, with $N=\rm{dim}_\cb X$, and
$F$ is a CR embedding if and only if, for all $1\leq j\leq N$,
$\opa_b f_j=0$.

A CR embedding is called
\emph{generic} if $\rm{dim}_\cb X- \rm{rk}_\cb H_{0,1}M=\rm{codim}_\rb
M$.

\section{Perturbation of CR structures}\label{s2}

In this section we shall define the notion of perturbation of a given CR structure on a manifold and introduce some new complex associated to the perturbed structure.

\begin{defin}\label{perturbation}
An almost CR structure $\widehat H_{0,1}\mb$ on $\mb$ is said to be \emph{a perturbation of finite distance} to a given CR structure $H_{0,1}\mb$ if
$\widehat H_{0,1}\mb$ can be represented as a graph in $T_\cb\mb$ over $H_{0,1}\mb$. It is called an
\emph{horizontal perturbation} of the CR structure $H_{0,1}\mb$ if
it can be represented as a graph in the complex
tangent bundle $H\mb=H_{1,0}\mb\oplus H_{0,1}\mb$ over $H_{0,1}\mb$,
which means that there exists $\Phi\in\cc_{0,1}(\mb,H_{1,0}\mb)$
such that
\begin{equation}\label{perturbation horizontale}
\widehat H_{0,1}\mb=\{\ol W\in T_\cb\mb~|~\ol W=\ol Z-\Phi(\ol Z),
\ol Z\in H_{0,1}\mb\}.
\end{equation}
\end{defin}

The horizontal perturbation $\widehat H_{0,1}\mb$ of the CR
structure $H_{0,1}\mb$ will be integrable if and only if given
$\ol L_1,\ol L_2\in \Gamma(\mb, H_{0,1}\mb)$ and $\ol L_i^\Phi=\ol L_i-\Phi(\ol L_i)$, $i=1,2$, there exists $\ol W$ such that $[\ol L_1^\Phi,\ol L_2^\Phi]=\ol W-\Phi(\ol W)=\ol W^\Phi$, which is equivalent to
\begin{equation}\label{integrability}
[\ol L_1^\Phi,\ol L_2^\Phi]^{1,0}=-\Phi([\ol L_1^\Phi,\ol L_2^\Phi]^{0,1}),
\end{equation}
if $X^{0,1}$ denotes the projection on $H_{0,1}\mb$ and $X^{1,0}$ the projection on $H_{1,0}\mb$ of $X\in H\mb$.

 Since
\begin{align*}
[\ol L_1^\Phi,\ol L_2^\Phi]&=[\ol L_1,\ol L_2]-[\Phi(\ol L_1),\ol L_2]-[\ol L_1,\Phi(\ol L_2)]+[\Phi(\ol L_1),\Phi(\ol L_2)]\\
&=[\ol L_1,\ol L_2]-\Phi([\ol L_1,\ol L_2])-([\ol L_1,\Phi(\ol L_2)]-[\ol L_2,\Phi(\ol L_1)]-\Phi([\ol L_1,\ol L_2]))\\
&+[\Phi(\ol L_1),\Phi(\ol L_2)]
\end{align*}
we get

$$[\ol L_1^\Phi,\ol L_2^\Phi]^{1,0}=-\Phi([\ol L_1,\ol L_2])-\opa_b\Phi(\ol L_1,\ol L_2)+[\Phi(\ol L_1),\Phi(\ol L_2)]$$

and the equation \eqref{integrability} is equivalent to
$$\opa_b\Phi(\ol L_1,\ol L_2)=[\Phi(\ol L_1),\Phi(\ol L_2)]-\Phi([\ol L_1,\Phi(\ol L_2)]^{0,1})+\Phi([\ol L_2,\Phi(\ol L_1)]^{0,1}),$$
which we simply write in the form
\begin{equation}\label{integrability2}
\opa_b^{H_{1,0}\mb}\Phi=\frac{1}{2}[\Phi,\Phi].
\end{equation}

Note that if $\widehat H_{0,1}\mb$ is an integrable horizontal
perturbation of $H_{0,1}\mb$, the space
$H\widehat\mb=H_{1,0}\widehat\mb\oplus H_{0,1}\widehat\mb$ coincides
with the space $H\mb$ and consequently the two abstract CR manifolds
$\mb$ and $\widehat\mb$ have the same characteristic bundle and
hence the same Levi form. This implies in particular that if $\mb$
satisfies condition Y(q) at each point, then $\widehat\mb$ satisfies
also condition Y(q) at each point and that if $\mb$ is $q$-concave
then $\widehat\mb$ is also $q$-concave.
\medskip

Assume $\mb$ is an abstract CR manifold  and $\widehat H_{0,1}\mb$
is an integrable horizontal perturbation of the original CR
structure $H_{0,1}\mb$ on $\mb$. Let $f$ be a complex valued function on $\mb$, then $f$ will be CR for the new structure $\widehat H_{0,1}\mb$ if and only if
$$\forall~ \ol L^\Phi\in\Gamma(\mb, \widehat H_{0,1}\mb),~ \ol L^\Phi f=0~\Leftrightarrow~\opa_b^\Phi f=df_{|_{\widehat H_{0,1}\mb}}=0.$$
But going back to the definition of $\widehat H_{0,1}\mb$, this means
$$\forall~ \ol L\in\Gamma(\mb, H_{0,1}\mb),~(\ol L -\Phi(\ol L))f=0~\Leftrightarrow~d''_\Phi f=(\opa_b -\Phi\lrcorner \pa_b)f=0,$$
where $\opa_b$ is the tangential Cauchy-Riemann operator associated to
the original CR structure $H_{0,1}\mb$ and $\Phi\lrcorner \pa_b f$ is the $(0,1)$-form for the initial structure defined by $\Phi\lrcorner \pa_b f(\ol L)=\Phi(\ol L)f$ for any section $\ol L\in\Gamma(\mb, H_{0,1}\mb)$.

Let us consider the two operators $$\opa_b^\Phi~:~\cc^1(\mb)\to\Gamma^0(\mb,\widehat H^{0,1}\mb)\quad {\rm and}\quad d''_\Phi~:\cc^1(\mb)\to\Gamma^0(\mb, H^{0,1}\mb),$$
then a fonction $f$ is CR for the new structure $\widehat H_{0,1}\mb$ if and only if $\opa_b^\Phi f=0$ or equivalently $d''_\Phi f=0$.

We are lead to consider two pseudo-complex~: $$(\ci_{0,*}(\widehat\mb),\opa_b^\Phi)\quad {\rm and}\quad (\ci_{0,*}(\mb),d''_\Phi),$$
where $\ci_{0,q}(\widehat\mb)=\Gamma(\mb,\bigwedge^q(\widehat H^{0,1}\mb))$ and $\ci_{0,q}(\mb)=\Gamma(\mb,\bigwedge^q(H^{0,1}\mb))$.
In degree $q\geq 1$, since the two structures $H^{0,1}\mb$ and $\widehat H^{0,1}\mb$ are integrable, the operators can be defined in the following way~: if $\alpha\in \ci_{0,q}(\widehat\mb)$ and $\ol L_1^\Phi,\dots,\ol L_{q+1}^\Phi\in\Gamma(\mb, \widehat H_{0,1}\mb)$, we set
\begin{align*}
\opa_b^\Phi\alpha(\ol L_1^\Phi,\dots,\ol L_{q+1}^\Phi)&=\sum_{j=1}^{q+1}(-1)^{j+1}\opa_b^\Phi(\alpha(\ol L_1^\Phi,\dots,\widehat{\ol L_j^\Phi},\dots,\ol L_{q+1}^\Phi))(\ol L_j^\Phi)\\&+\sum_{j<k}(-1)^{j+k}\alpha([\ol L_j^\Phi,\ol L_k^\Phi],\ol L_1^\Phi,\dots,\widehat{\ol L_j^\Phi},\dots,\widehat{\ol L_k^\Phi},\dots,\ol L_{q+1}^\Phi)
\end{align*}
and if $\beta\in \ci_{0,q}(\mb)$ and $\ol L_1,\dots,\ol L_{q+1}\in \Gamma(\mb,H_{0,1}\mb)$, we set
\begin{align*}
d''_\Phi\beta(\ol L_1,\dots,\ol L_{q+1})&=\sum_{j=1}^{q+1}(-1)^{j+1}d''_\Phi(\beta(\ol L_1,\dots,\widehat{\ol L_j},\dots,\ol L_{q+1}))(\ol L_j)\\&+\sum_{j<k}(-1)^{j+k}\beta([\ol L_j,\ol L_k],\ol L_1,\dots,\widehat{\ol L_j},\dots,\widehat{\ol L_k},\dots,\ol L_{q+1}).
\end{align*}

Note that $\opa_b^\Phi=d{|_{\widehat H_{0,1}\mb\times\dots\times \widehat H_{0,1}\mb}}$ and the pseudo-complex $(\ci_{0,*}(\widehat\mb),\opa_b^\Phi)$ is a differential complex, which is nothing else than the tangential Cauchy-Riemann complex on $\mb$ associated to the new CR structure $\widehat H^{0,1}\mb$.

\begin{lem}
The pseudo-complex $(\ci_{0,*}(\mb),d''_\Phi)$ is a differential complex, i.e. $d''_\Phi\circ d''_\Phi =0$ if and only if $\Phi$ satisfies  $[\ol L_1^\Phi,\ol L_2^\Phi]=[\ol L_1,\ol L_2]^\Phi$, for any $\ol L_1,\ol L_2\in \Gamma(\mb, H_{0,1}\mb)$.
\end{lem}
\begin{proof}
Let $f\in\ci(\mb)$ be a $\ci$-smooth function on $\mb$, going back to the definition of the operator $d''_\Phi$ we get
\begin{align*}
d''_\Phi\circ d''_\Phi f(\ol L_1,\ol L_2)&=\ol L_1^\Phi\ol L_2^\Phi f-\ol L_2^\Phi\ol L_1^\Phi f-[\ol L_1,\ol L_2]^\Phi f\\
&=\opa_b^\Phi\circ\opa_b^\Phi f(\ol L_1^\Phi,\ol L_2^\Phi)+([\ol L_1^\Phi,\ol L_2^\Phi]-[\ol L_1,\ol L_2]^\Phi)f\\
&=([\ol L_1^\Phi,\ol L_2^\Phi]-[\ol L_1,\ol L_2]^\Phi)f,
\end{align*}
which proves the lemma.
\end{proof}

\section{Embedding of small horizontal perturbations in complex manifolds}\label{s3}

Let $(\mb,H_{0,1}\mb)$ be an abstract compact CR manifold of class
$\ci$ and $\ec_0~:~\mb\to M_0\subset X$ be a $\ci$-smooth CR
embedding in a complex manifold $X$, then $M_0$ is a compact CR
submanifold of $X$ of class $\ci$ with the CR structure
$H_{0,1}M_0=d\ec_0(H_{0,1}\mb)=T_\cb M_0\cap T_{0,1} X$ and the
tangential Cauchy-Riemann operator $\opa_b$.

Let $\widehat H_{0,1}\mb$ be an horizontal perturbation of
$H_{0,1}\mb$, we are looking for an embedding $\ec~:~\mb\to
M\subset X$ of class $\cc^l$, $l\geq 1$, such that $d\ec(\widehat
H_{0,1}\mb)=T_\cb M\cap
T_{0,1} X$, i.e. $\ec$ is a CR embedding.

Set $\widehat H_{0,1}M_0=d\ec_0(\widehat H_{0,1}\mb)$, as $\ec_0$ is
a CR embedding then $\widehat H_{0,1}M_0$ is an horizontal
perturbation of $H_{0,1}M_0$ and consequently it is defined by a
$(0,1)$-form $\Phi\in\cc^k_{0,1}(M_0,H_{1,0}M_0)$, $k\geq 1$. We
denote by $\opa_b^\Phi$ the associated tangential Cauchy-Riemann
operator.

We will consider only \emph{small} (the sense will be precised
later) perturbations of the original structure, thus it is
reasonable to assume that the diffeomorphism
$F=\ec\circ\ec_0^{-1}~:~M_o\subset X\to M\subset X$ is close to
identity.

We equip the manifold $X$ with some Riemannian metric (for example,
if $X=\cb P^n$, take the Fubini-Study metric). The idea
is to look for some $F$ in the subset of the restrictions to $M_0$
of $\cc^l$-diffeomorphisms of $X$ parametrized by sections of the
vector bundle $TX$ by mean of the exponential map. Let us consider
the following diagram where $U$ is a neighborhood of the zero
section and $\sigma$ a section of $TX$ over $M_0$
$$
\begin{array}{ccc}
U\subset TX & \overset\exp\longrightarrow & X\times X\\
(x,\sigma(x))& \longmapsto & (x,F(x))=\exp_x\sigma(x)\\
\big\uparrow& &\big\downarrow\\
M_0 & \longrightarrow & X\\
x & \longmapsto & F(x)
\end{array}
$$
In fact $\ec$ will be a CR embedding if and only if $F$ is CR as a map from $\widehat M_0=(M_0,\widehat H_{0,1}M_0)$ into $X$,
which means that we have to find a section $\sigma$ of $TX$
over $M_0$ such that the image of the new CR structure $\widehat
H_{0,1}M_0$ by the tangent map to $\exp\circ~\sigma$ is contained in $T_{0,1}
(X\times X)$.

More precisely, using that for all $x\in M_0$ we can write $F(x)=\pi(x,F(x))=\pi(\exp_x \sigma(x))$, with $\pi$ the second projection from $X\times X$ onto $X$, the map $F$ will be CR if and only if for all vector fields $\ol L^\Phi=\ol L-\Phi(\ol L)\in\Gamma(\mb,\widehat H^{0,1}\mb)$,  $\ol L\in \Gamma(\mb, H^{0,1}\mb)$,
we have
$$dF(\ol L^\Phi)=d(\pi(\exp_x \sigma(x))(\ol L^\Phi)=0.$$
As the differential of the map $\exp$ at a point is given by
the map $(u,\xi)\mapsto (u,u+\xi)$, this is equivalent to $(d(Id)+d(\sigma))(\ol L^\Phi)=0$, i.e.
$$\opa_b^\Phi\sigma=-\opa_b^\Phi Id.$$

Since $d(Id)(\ol L^\Phi)=-\Phi(\ol L)$, all that means that we have to solve the equation
$$d''_\phi\sigma=\Phi.$$

The remaining of the section will be devoted to the proof of the following theorem:
\begin{thm}\label{perturb-var}
Let $(\mb,H_{0,1}\mb)$ be an abstract compact CR
manifold of class $\cc^\infty$, which is  smoothly embeddable as a
CR manifold in a complex manifold $X$. Assume that $(\mb,H_{0,1}\mb)$
is $2$-concave. Let $\ec_0$ be a smooth
embedding  from $\mb$ into $X$ and denote by $M_0$ the image of $\mb$
by the embedding $\ec_0$. Assume the $\opa_b$-group of
cohomology $H^{0,1}(M_0,T_{1,0}X_{|_{M_0}})=0$ and let $\widehat
H_{0,1}\mb$ be an horizontal perturbation of $H_{0,1}\mb$ defined by
a $(0,1)$-form $\Phi\in\cc_{0,1}^{l+3}(\mb,H_{1,0}\mb)$, $l\geq 1$.

Then there exists a positive real number $\delta$ such that if
$\|\Phi\|_{l+3}<\delta$, then the CR manifold $(\mb,\widehat
H_{0,1}\mb)$ is embeddable in $X$ as a CR submanifold of class
$\cc^l$.
\end{thm}

\subsection{Reduction to a fixed point theorem}\label{s2.1}

Let $E$ be a CR bundle over $\mb$ which satisfies $H^{0,1}(\mb,E)=0$
and $g$ be a $\opa_b^\Phi$-closed $(0,1)$-form in
$\cc^{l+2}_{0,1}(\widehat \mb,E)$ and let us consider the following equation:
\begin{equation}\label{dbarphi}
\opa_b^\Phi v=g.
\end{equation}

By definition of the $\opa_b^\Phi$-operator, the equation
$\opa_b^\Phi v=g$ is equivalent to the equation
\begin{equation}\label{dbarbis}
\opa_b v=\wt g+\Phi\lrcorner\pa_b v,
\end{equation}
where $\wt g$ is the $(0,1)$-form on $\mb$ relatively to the initial structure $H_{0,1}\mb$ defined by $\wt g(\ol L)=g(\ol L-\Phi(\ol L))$ for $\ol L\in \Gamma(\mb,H_{0,1}\mb)$.

A natural tool to solve such an equation is a global homotopy
formula for the $\opa_b$-operator with good estimates.

Assume $\mb$ is $2$-concave, then by \cite{NaPo}, since $\mb$ is embeddable and $1$-concave, $\mb$ is locally generically embeddable and we may
apply the results in \cite{BaLa} and \cite{LaLeglobal} on local
estimates and global homotopy formulas for the tangential Cauchy-Riemann
operator.

In \cite{BaLa} the following result is proved

\begin{prop}\label{estmeeslocales}
Let $M$ be a $2$-concave CR generic
submanifold of $X$ of class $\ci$.
For each point in $M$, there exist
a neighborhood $U$ and linear operators
\begin{equation*}
T_r~:~\cc^0_{n,r}(M)\to \cc^0_{n,r-1}(U),\quad 1\leq r\leq 2
\end{equation*}
with the following two properties~:

(i) For all $l\in\nb$ and $1\leq r\leq 2$,
\begin{equation*}
T_r(\cc^l_{n,r}(M))\subset \cc^{l+1/2}_{n,r-1}(\ol U)
\end{equation*}
and $T_r$ is continuous as an operator between $\cc^l_{n,r}(M)$ and
$\cc^{l+1/2}_{n,r-1}(\ol U)$.

(ii) If $f\in\cc^1_{n,r}(M)$, $0\leq r\leq 1$, has compact support in
$U$, then, on $U$,
\begin{equation}\label{homotopielocale}
f=\begin{cases}T_{1}\opa_b f&\text{if
}r=0\,,\\\opa_b T_1f+T_{2}\opa_b f\qquad&\text{if }r=1 \,.
\end{cases}
\end{equation}
\end{prop}
{\parindent=0pt and in \cite{LaLeglobal} we have derived from the
  previous proposition
a global homotopy formula by mean of a functional analytic
construction.}

\begin{thm}\label{formuleglobale}
Let $E$ be an holomorphic vector bundle over $X$ and  $\mb$ be a
compact $2$-concave locally generically embeddable CR submanifold of $X$ of class $\ci$ such that
$H^{0,1}(\mb,E)=0$. Then there exist continuous linear operators
\begin{equation*}
A_r~:~\cc^0_{0,r}(\mb,E)\to \cc^0_{0,r-1}(\mb,E),\quad 1\leq r\leq 2
\end{equation*}
such that

(i) For all $l\in\nb$ and $1\leq r\leq 2$,
\begin{equation*}
A_r(\cc^l_{0,r}(\mb,E))\subset \cc^{l+1/2}_{0,r-1}(\mb,E)
\end{equation*}
and $A_r$ is continuous as an operator between $\cc^l_{0,r}(\mb,E)$
and $\cc^{l+1/2}_{0,r-1}(\mb,E)$.

(ii) For all $f\in\cc^1_{0,1}(\mb,E)$
\begin{equation}\label{homotopieglobale}
f=\opa_b A_1f+A_{2}\opa_b f
\end{equation}
\end{thm}

In fact we need better estimates  than the previous ones to reduce
the solvability of our equation \ref{dbarbis} to a fixed point
theorem.

Assume there exist Banach spaces $\bc^l(\mb)$, $l\in\nb$, with the
following properties~:

(i) $\cc^{l+2}(\mb)\subset \bc^{2l+1}(\mb)\subset
\bc^{2l}(\mb)\subset\cc^l(\mb)$ ;

(ii) $\cap_{l\in\nb} \bc^l(\mb)=\ci(\mb)$ ;

(iii) $\bc^l(\mb)$ is invariant under horizontal perturbations of
the CR structure

(iv) If $f\in\bc^l(\mb)$, $l\geq 1$, $X_\cb f\in\bc^{l-1}(\mb)$ when
$X_\cb$ is a complex vector field tangent to $\mb$

and that the previous operators $A_r$, $r=1,2$ are linear continuous
operators from $\bc^l_{0,r}(\mb,E)$ into $\bc^{l+1}_{0,r-1}(\mb,E)$

Since $\mb$ is $2$-concave and $E$ satisfies $H^{0,1}(\mb,E)=0$,
under this additional hypothesis, if $v$ is a solution of
(\ref{dbarphi}), then $\opa_b (\wt g+\Phi\lrcorner\pa_b v)=0$ and by
\eqref{homotopieglobale}
$$\opa_b(A_{1}(\wt g+\Phi\lrcorner\pa_b v))=\wt g+\Phi\lrcorner\pa_b v.$$

Assume $\Phi$ is of class $\cc^{l+2}$, then the map
\begin{align*}
\Theta~:~\bc^{2l+1}(\mb,E)&\to\bc^{2l+1}(\mb,E)
\\
v&\mapsto A_{1} \wt g+A_{1}(\Phi\lrcorner\pa_b v)\,.
\end{align*}
is continuous, and the fixed points of $\Theta$ are good candidates
to be solutions of (\ref{dbarphi}).

\subsection{A fixed point theorem}\label{s2.2}
In this section we assume that all the assumptions of the previous
section are satisfied.

Let $\delta_0$ such that, if $\|\Phi\|_{l+2}<\delta_0$, then the
norm of the bounded endomorphism $A_{1}\circ\Phi\lrcorner\pa_b$ of
$\bc^{2l+1}(\mb,E)$ is equal to $\epsilon_0<1$. We shall prove that,
if $\|\Phi\|_{l+2}<\delta_0$, the map $\Theta$ admits a unique fixed
point, which is a solution of the equation $\opa_b^\Phi v=g$.

Consider first the uniqueness of the fixed point. Assume $v_1$ and
$v_2$ are two fixed points of $\Theta$, then
\begin{align*}
v_1&=\Theta(v_1)=A_{1}\wt g+A_{1}(\Phi\lrcorner\pa_b v_1)\\
v_2&=\Theta(v_2)=A_{1}\wt g+A_{1}(\Phi\lrcorner\pa_b v_2).
\end{align*}
This implies
$$v_1-v_2=A_{1}\big(\Phi\lrcorner\pa_b (v_1-v_2)\big)$$
and, by the hypothesis on $\Phi$,
$$\|v_1-v_2\|_{\bc^{2l+1}}<\|v_1-v_2\|_{\bc^{2l+1}}$$
or $v_1=v_2$ and hence $v_1=v_2$.

For the existence we proceed by iteration. We set
$v_0=\Theta(0)=A_{1}(\wt g)$ and, for $n\geq 0$, $v_{n+1}=\Theta(v_n)$.
Then for $n\geq 0$, we get
$$v_{n+1}-v_n=A_{1}(\Phi\lrcorner\pa_b (v_n-v_{n-1})).$$

Therefore, if $\|\Phi\|_{l+2}<\delta_0$, the sequence
$(v_n)_{n\in\nb}$ is a Cauchy sequence in the Banach space
$\bc^{2l+1}(\mb,E)$ and hence converges to a form $v$, moreover by
continuity of the map $\Theta$, $v$ satisfies $\Theta(v)=v$.

It remains to prove that $v$ is a solution of (\ref{dbarphi}). Since
$H^{0,1}(\mb,E)=0$, it follows from \eqref{homotopieglobale} and
from the definition of the sequence $(v_n)_{n\in\nb}$ that
\begin{equation*}
\wt g-d''_\Phi
v_{n+1}=\Phi\lrcorner\pa_b(v_{n+1}-v_n)+A_{2}\opa_b(\wt g+\Phi\lrcorner\pa_b
v_n).
\end{equation*}
For any $\ol L_1, \ol L_2\in \Gamma(\mb,H_{0,1}\mb)$,
\begin{align*}
\opa_b(\wt g+\Phi\lrcorner\pa_b v_n)(\ol L_1, \ol L_2)&=(\opa_b \wt g-\opa_b(\opa_b-\Phi\lrcorner\pa_b) v_n)(\ol L_1, \ol L_2)\\
&=\opa_b (\wt g-d''_\Phi v_n)(\ol L_1, \ol L_2)
\end{align*}
Let us consider the $(0,1)$-form $u$ for the new structure defined by $u=g-\opa_b^\Phi v_n$, the associated $(0,1)$-form $\wt u$ for the initial structure satisfies $\wt u=\wt g-d''_\Phi v_n$. Note that, since $(\opa_b^\Phi)^2=0$ and $\opa_b^\Phi g=0$, we have $\opa_b^\Phi u=0$ and this implies
\begin{align*}
\opa_b\wt u(\ol L_1,\ol L_2)&=\ol L_1\wt u(\ol L_2)-\ol L_2\wt u(\ol L_1)-\wt u([\ol L_1,\ol L_2])\\
&=\ol L_1^\Phi u(\ol L_2^\Phi)-\ol L_2^\Phi u(\ol L_1^\Phi)+\Phi(\ol L_1)u(\ol L_2^\Phi)-\Phi(\ol L_2)u(\ol L_1^\Phi)-\wt u([\ol L_1,\ol L_2])\\
&=\opa_b^\Phi u(\ol L_1^\Phi,\ol L_2^\Phi)+u([\ol L_1^\Phi,\ol L_2^\Phi])-u([\ol L_1,\ol L_2]-\Phi([\ol L_1,\ol L_2]))\\
&+\Phi(\ol L_1)u(\ol L_2^\Phi)-\Phi(\ol L_2)u(\ol L_1^\Phi)\\
&=u([\ol L_1^\Phi,\ol L_2^\Phi])-[\ol L_1,\ol L_2]+\Phi([\ol L_1,\ol L_2])\\
&+\Phi(\ol L_1)u(\ol L_2^\Phi)-\Phi(\ol L_2)u(\ol L_1^\Phi)
\end{align*}

Since both structures $H_{0,1}\mb$ and $\widehat H_{0,1}\mb$ are integrable, the vector field $[\ol L_1^\Phi,\ol L_2^\Phi])-[\ol L_1,\ol L_2]+\Phi([\ol L_1,\ol L_2]$ is a section of $\widehat H_{0,1}\mb$, hence there exists $\ol W_\Phi(\ol L_1,\ol L_2)$ such that 
$$[\ol L_1^\Phi,\ol L_2^\Phi])-[\ol L_1,\ol L_2]+\Phi([\ol L_1,\ol L_2]=\ol W_\Phi(\ol L_1,\ol L_2)-\Phi(\ol W_\Phi(\ol L_1,\ol L_2)).$$
Moreover
$$[\ol L_1^\Phi,\ol L_2^\Phi])-[\ol L_1,\ol L_2]+\Phi([\ol L_1,\ol L_2]=[\Phi(\ol L_1),\Phi(\ol L_2)]-\opa_b\Phi(\ol L_1,\ol L_2)-\Phi(\ol L_1)\ol L_2+\Phi(\ol L_2)\ol L_1)$$ and $\ol W_\Phi(\ol L_1,\ol L_2)$ depends on $\Phi$ at the order $1$.
 
Finally 
$$\opa_b\wt u(\ol L_1,\ol L_2)=\wt u(\ol W_\Phi(\ol L_1,\ol L_2))+\Phi(\ol L_1)\wt u(\ol L_2)-\Phi(\ol L_2)\wt u(\ol L_1)$$
and going back to $u=g-\opa_b^\Phi v_n$ we get
\begin{equation}\label{erreur}
\wt g-d''_\Phi
v_{n+1}=\Phi\lrcorner\pa_b(v_{n+1}-v_n)+A_{2}(\tau_\Phi(\wt g-d''_\Phi
v_n)),
\end{equation}
where $\tau_\Phi$ is defined so that $\tau_\Phi(\wt u)(\ol L_1, \ol L_2)=\wt u(\ol W_\Phi(\ol L_1,\ol L_2))+\Phi(\ol L_1)\wt u(\ol L_2)-\Phi(\ol L_2)\wt u(\ol L_1)$.

Note that since $g\in\cc^{l+2}_{0,1}(\mb,E)$ and $\Phi$ is of class
$\cc^{l+2}$, then $v_0\in\bc^{2l+1}(\mb)$ and $\opa_b
v_0-\Phi\lrcorner\pa_b v_0=d''_\Phi v_0$ is in
$\bc^{2l}_{0,1}(\mb,E)$, moreover, by definition of an horizontal perturbation, for any section $\ol L$ of  $H_{0,1}\mb$  the vector field $\Phi(\ol L)$ is complex tangent.

Thus it follows by induction that $d''_\Phi
v_n\in\bc^{2l}_{0,1}(\mb,E)$ for all $n\in\nb$ and, if $\Phi$ is of class
$\cc^{l+3}$, by (\ref{erreur}) we have the
estimate
\begin{equation}\label{est1}
\|\wt g-d''_\Phi v_{n+1}\|_{\bc^{2l}}
\leq\|\Phi\lrcorner\pa_b\|\|(v_{n+1}-v_n)\|_{\bc^{2l+1}}+\|A_{2}\|\|\wt g-d''_\Phi
v_n\|_{\bc^{2l}}\|\Phi\|_{l+3}.
\end{equation}
Let $\delta$ such that if $\|\Phi\|_{l+3}<\delta$, then the maximum
$\|A_1\circ\Phi\lrcorner\pa_b\|$ and $\|\Phi\|_{l+3}\|A_2\|$ is equal to $\epsilon<1$. Assume
$\|\Phi\|_{l+3}<\delta$, then by induction we get
\begin{equation}\label{est2}
\|\wt g-d''_\Phi v_{n+1}\|_{\bc^{2l}} \leq
(n+1)\epsilon^{n+1}\|\Phi\lrcorner\pa_b\|\|v_0\|_{\bc^{2l+1}}+\epsilon^{n+1}\|\wt g-d''_\Phi
v_0\|_{\bc^{2l}}.
\end{equation}
But $\wt g-d''_\Phi v_0=\Phi\lrcorner\pa_b
A_{1}\wt g+A_{2}(\tau_\Phi(\wt g))$ and hence $\|\wt g-d''_\Phi
v_0\|_{\bc^{2l}}\leq\|\Phi\lrcorner\pa_b\|\|A_{1}\wt g\|_{\bc^{2l+1}}
+\epsilon \|\wt g\|_{\bc^{2l}}$. This implies
\begin{equation}\label{est3}
\|\wt g-d''_\Phi v_{n+1}\|_{\bc^{2l}}\leq
(n+2)\epsilon^{n+1}\|\Phi\lrcorner\pa_b\|\|A_{1}\|\|\wt g\|_{\bc^{2l+1}}+\epsilon^{n+2}\|\wt g\|_{\bc^{2l}}.
\end{equation}
Since $\epsilon<1$, the righthand side of (\ref{est3}) tends to
zero, when $n$ tends to infinity and by continuity of the operator
$d''_\Phi$ from $\bc^{2l+1}(\mb,E)$ into
$\bc^{2l}_{0,1}(\mb,E)$, we get that $v$ is a solution of
(\ref{dbarphi}).

\subsection{Solution of the embedding
problem}\label{s2.3}

In the setting of the beginning of section \ref{s2}, let us consider
the following anisotropic H\"older spaces of functions:

- $\ac^\alpha(M_0)$, $0<\alpha<1$, is the set of continuous
functions on $M_0$ which are in $\cc^{\alpha/2}(M_0)$.

- $\ac^{1+\alpha}(M_0)$, $0<\alpha<1$, is the set of functions $f$
such that $f\in\cc^{(1+\alpha)/2}(M_0)$ and $X_\cb
f\in\cc^{\alpha/2}(M_0)$, for all complex tangent vector fields
$X_\cb$ to $M_0$.  Set
\begin{equation}\label{holderanis}
  \|f\|_{A\alpha}=\|f\|_{(1+\alpha)/2}+\sup_{\|X_\cb\|\leq 1}\|X_\cb f\|_{\alpha/2}
\end{equation}

- $\ac^{l+\alpha}(M_0)$, $l\geq 2$, $0<\alpha<1$, is the set of
functions $f$ of class $\cc^{[l/2]}$ such that
$Xf\in\ac^{l-2+\alpha}(M_0)$, for all tangent vector fields $X$ to
$M$ and $X_\cb f\in\ac^{l-1+\alpha}(M_0)$, for all complex tangent
vector fields $X_\cb$ to $M_0$.

Fix some $0<\alpha<1$ and set $\bc^l(M_0)=\ac^{l+\alpha}(M_0)$.
This sequence $(\bc^l(M_0), l\in\nb)$ is a sequence of Banach spaces
which satisfies properties (i) to (iv) listed in section \ref{s2.1}

Moreover it is proved in \cite{Laanis} that the  operators $A_r$,
$r=1,2$, from Theorem \ref{formuleglobale}  are linear continuous
operators between the anisotropic H\"older spaces
$\bc^l_{0,r}(M_0,E)$ and $\bc^{l+1}_{0,r-1}(M_0,E)$.

One can also consider the anisotropic H\"older
spaces introduced by Folland and Stein when they studied the
tangential Cauchy-Riemann complex on the
Heisenberg group and more generally on strictly pseudoconvex CR
manifolds.

Let $M$ be a generic $CR$ manifold of class $\cc^\infty$ of real
dimension $2n$ and CR dimension $2n-k$
and $D$ be a relatively compact domain in $M$.  Let
$X_1,\dots,X_{2n-2k}$ be a real basis of
$HM$. A $\cc^1$ curve
$\gamma~:~[0,r]\to M$ is called admissible if for every $t\in [0,r]$,
\begin{equation*}
\frac{d\gamma}{dt}(t)=\sum_{j=1}^{2n-2k}c_j(t)X_j(\gamma(t))
\end{equation*}
where $\sum|c_j(t)|^2\leq 1$.

The Folland-Stein anisotropic
H\"older spaces
$\Gamma^{p+\alpha}(\ol D\cap M)$  are defined in the
following way:

- $\Gamma^\alpha(\ol D\cap M)$, $0<\alpha<1$, is the set of
continuous fonctions in $\ol D\cap M$ such that if for every
$x_0\in\ol D\cap M$
$$\sup_{\gamma(.)}\frac{|f(\gamma(t)-f(x_0)|}{|t|^\alpha}<\infty$$
for any admissible complex tangent curve $\gamma$ through $x_0$.

- $\Gamma^{p+\alpha}(\ol D\cap M)$, $p\ge 1$, $0<\alpha<1$, is the
set of continuous fonctions in $M$ such that $X_\cb
f\in\Gamma^{p-1+\alpha}(\ol D\cap M)$, for all complex tangent
vector fields $X_\cb$ to $M$.

Fix some $0<\alpha<1$, the sequence $(\Gamma^{p+\alpha}, p\in\nb)$ is
also a sequence of Banach spaces
which satisfies properties (i) to (iv) listed in section \ref{s2.1}

Continuity properties for the operators $A_r$ and $B_r$, $r=1,2$,
defined in Theorem
\ref{formuleglobale} are proved in
\cite{Laanis}. More precisely,
for all $p\in\nb$ and $0<\alpha<1$, the
operators $A_r$, $r=1,2$, from Theorem
\ref{formuleglobale} are continuous from
$\Gamma^{p+\alpha}_{0,r}(M)$ into $\Gamma_{0,r-1}^{p+1+\alpha}(M)$.
\medskip

We can apply now the method developed in section \ref{s2.2} with
$E=T_{1,0}X_{|_{M_0}}$ to solve the equation
$$\opa_b^\Phi \sigma=-\opa_b^\Phi Id$$
and a unique solution is given by a fixed point $\sigma$ of the map $\Theta$
if $\|\Phi\|_{l+3}$ is sufficiently small, moreover this fixed point is contained in the neighborhood $U$ of the zero section in $TX$ on which the exponential map is defined, once again if $\|\Phi\|_{l+3}$ is small enough, since $\|\sigma\|_l<\|\Phi\|_{l+2}$. We deduce that
$\ec=F\circ\ec_0$, with $F$ defined by $F(x)=\exp_x\sigma(x)$, is
the embedding we are looking for, which ends the proof of Theorem
\ref{perturb-var}.

\section{Stability of embeddability in $\cb P^N$}\label{s4}

In this section we will consider the case when the compact CR manifold $(\mb, H_{0,1}\mb)$ is embeddable in some $\cb P^N$. Let $\ec_0~:~\mb\to M_0\subset \cb P^N$ be a $\ci$-smooth CR embedding, it is defined by some homogeneous coordinates $(f_1,\dots,f_N)$, each $f_j$, $j=1,\dots,N$, being a CR function such that on the set $U_j=\{x\in\mb~|~f_j\neq 0\}$, $(\frac{f_1}{f_j},\dots,\frac{f_{j-1}}{f_j},\frac{f_{j+1}}{f_j},\dots,\frac{f_N}{f_j})$ defines a diffeomorphism.

Let us consider an horizontal perturbation $\widehat H_{0,1}\mb$ of $H_{0,1}\mb$, we are looking for a CR embedding $\ec$ of $(\mb,\widehat H_{0,1}\mb)$ in $\cb P^N$. As we will consider only small perturbations of the original structure, we are lead to look for some $\ec$ given by some small perturbation $(f_1-g_1,\dots,f_N-g_N)$ of the original homogeneous coordinates $(f_1,\dots,f_N)$, which have to satisfy $\opa_b^\Phi (f_j-g_j)=0$ for all $j=1,\dots,N$. Using the definition of the operator  $\opa_b^\Phi$, this is equivalent  to
\begin{equation}\label{dbar-proj}
\opa_b^\Phi g_j=\Phi\lrcorner \pa_b f_j\quad {\rm for~all}~ j=1,\dots,N.
\end{equation}
Note that the second member of these equations is controlled by the form $\Phi$ which defines the perturbation, so that if we can solve the equation
$\opa_b^\Phi g=f$ with $\cc^l$-estimates uniformly with respect to $\Phi$, then for $\Phi$ sufficiently small the homogeneous coordinates $(f_1-g_1,\dots,f_N-g_N)$ will define the CR embedding we are looking for.

\subsection{Solving the $\opa_b^\Phi$ with estimates}

 We are interested in solving the equation
\begin{equation}\label{dbar-perturb}
\opa_b^\Phi \alpha=\beta
\end{equation}
with estimates, uniformly with respect to $\Phi$, in bidegree $(0,1)$ for an horizontal perturbation, given by a form $\Phi$, of the CR structure of an abstract compact CR
manifold $(\mb, H_{0,1}\mb)$ of class $\cc^\infty$,
 when $\Phi$ is sufficiently small. We will follow the method used by \cite{ShWa} leading to a strong Hodge decomposition theorem and homotopy formulas.

We equip $\mb$ with a hermitian metric such that $H_{0,1}\mb$ and $H_{1,0}\mb$ are orthogonal and we consider the Kohn-Laplacian
$$\Box_b=\opa_b\opa_b^*+\opa_b^*\opa_b.$$

For $\mb$ equipped with the perturbed structure $\widehat H_{0,1}\mb$, we can also consider the Kohn-Laplacian
$$\Box_b^\Phi=\opa_b^\Phi{\opa_b^\Phi}^*+{\opa_b^\Phi}^*\opa_b^\Phi$$
replacing the $\opa_b$-operator by the new operator $\opa_b^\Phi$.
We first establish some {\sl a priori} estimates for $\Box_b^\Phi$.
\begin{thm}\label{apriori}
Let $(\mb, H_{0,1}\mb)$ be an abstract compact CR
manifold of class $\cc^\infty$ and $\widehat
H_{0,1}\mb$ be an horizontal perturbation of $H_{0,1}\mb$ defined by
a $(0,1)$-form $\Phi\in\cc_{0,1}^{l}(\mb,H_{1,0}\mb)$, $l\geq 1$, such that $\|\Phi\|_{\cc^0}<1$. Assume that the Levi form of $\mb$ satisfies condition $Y(1)$ at $x_0\in\mb$. Then there exists $\delta>0$ and a sufficiently small neighborhood $U$ of $x_0$ such that, if $\|\Phi\|_{\cc^1}<\delta$,
$$\|\alpha\|^2_{\frac{1}{2}}\leq C(1+K\|\Phi\|_{\cc^1})(\|\opa_b^\Phi\alpha\|^2+\|{\opa_b^\Phi}^*\alpha\|^2+\|\alpha\|^2)$$
for all $\alpha\in\dc_{0,1}(U)$, where $\|.\|$ and $\|.\|_s$ are respectively the $L^2$ and the Sobolev norms.
\end{thm}
\begin{proof}
Let $\ol L_1,\dots,\ol L_{n-k}$ be an orthonormal basis for $H_{0,1}\mb$ and $\ol\omega_1,\dots,\ol\omega_{n-k}$ its dual basis. Since $\|\Phi\|_{\cc^0}<1$, then $\ol L_1^\Phi,\dots,\ol L_{n-k}^\Phi,L_1^\Phi,\dots, L_{n-k}^\Phi$, where $L_j^\Phi=L_j-\ol{\Phi(\ol L_j)}$, $j=1,\dots,n-k$, defines a basis of $H\mb$. We denote by $\ol\omega_1^\Phi,\dots,\ol\omega_{n-k}^\Phi$ the dual basis of $\ol L_1^\Phi,\dots,\ol L_{n-k}^\Phi$, then, for all $j=1,\dots,n-k$, we have $\ol\omega_j^\Phi=\ol\omega_j+(^t(I-\Phi)^{-1}-I)(\ol\omega_j)=\ol\omega_j+\theta_j^\Phi$.

If $f$ is a smooth function on $\mb$, we set
$$\|f\|_\lc^2=\sum_{i=1}^{n-k}\|L_i f\|^2+\|f\|^2,\quad \|f\|_{\ol\lc}^2=\sum_{i=1}^{n-k}\|\ol L_i f\|^2+\|f\|^2.$$

Let $\alpha$ be a $(0,r)$-form for the new structure $\widehat H_{0,1}\mb$ with smooth coefficients, then
$$\alpha=\sum_{|I|=r}\alpha_I~\ol\omega^\Phi_I,$$
where $\ol\omega^\Phi_I=\ol\omega^\Phi_{i_1}\wedge\dots\wedge\ol\omega^\Phi_{i_r}$ if $I$ is the multi-indice $I=(i_1,\dots,i_r)$ and
we have
$$\opa_b^\Phi\alpha=\sum_{j=1}^{n-k}\sum_{|I|=r}\ol L_j^\Phi~\alpha_I~\ol\omega^\Phi_j\wedge\ol\omega^\Phi_I+\dots,$$
where $\dots$ denote terms which do not involve the derivatives of $\alpha$.

If $r=0$, then
\begin{align*}
\opa_b^\Phi\alpha&=\sum_{j=1}^{n-k}(\ol L_j-\Phi(\ol L_j))\alpha~ \ol\omega^\Phi_j\\
&=\opa_b\alpha+\sum_{j=1}^{n-k}[\ol L_j~\alpha~ \ol\theta^\Phi_j-\Phi(\ol L_j)~\alpha~ \ol\omega^\Phi_j]
\end{align*}
and since the perturbation is horizontal, i.e. $\Phi\in\cc_{0,1}^{l}(\mb,H_{1,0}\mb)$, we get
$$\|\opa_b\alpha\|^2\leq
\|\opa_b^\Phi\alpha\|^2+O(\|\Phi\|_{\cc^0}(\|\alpha\|_\lc^2+\|\alpha\|_{\ol\lc}^2)).$$

If $r=1$, then
\begin{align*}
\opa_b^\Phi\alpha&=\sum_{i,j=1}^{n-k}(\ol L_j-\Phi(\ol L_j))\alpha_i~ \ol\omega^\Phi_j\wedge\ol\omega^\Phi_i+\dots\\
&=\opa_b\alpha+\sum_{i,j=1}^{n-k}(\ol L_j~\alpha_i~ \ol\theta^\Phi_j\wedge\ol\omega^\Phi_i-\Phi(\ol L_j)~\alpha_i~ \ol\omega^\Phi_j\wedge\ol\omega^\Phi_i)+R(\Phi,\alpha),
\end{align*}
where $R(\Phi,\alpha)$ is controlled by $(1+\|\Phi\|_{\cc^1})\|\alpha\|$,
and if $(\ol L_j^\Phi)^*$ denotes the Hilbert adjoint of $\ol L_j^\Phi$
\begin{align*}
{\opa_b^\Phi}^*\alpha&=\sum_{i,j=1}^{n-k}(\delta_i^j+\lambda_{ij}(\Phi))(\ol L_j^\Phi)^*\alpha_i+R^*(\Phi,\alpha)\\
&=\sum_{i,j=1}^{n-k}(\delta_i^j+\lambda_{ij}(\Phi))(-L_j-(\Phi(\ol L_j))^*)\alpha_i+R^*(\Phi,\alpha)\\
&=\opa_b^*\alpha-\sum_{i,j=1}^{n-k}(\Phi(\ol L_j))^*~\alpha_j+\lambda_{ij}(\Phi)(-L_j-(\Phi(\ol L_j))^*)\alpha_i+R^*(\Phi,\alpha),
\end{align*}
where $\lambda_{ij}(\Phi)$ is controlled by $\|\Phi\|_{\cc^0}$, $\|(\Phi(\ol L_j))^*f\|=O(\|\Phi\|_{\cc^1}\|f\|_{\ol\lc})$ and $R^*(\Phi,\alpha)$ is controlled by $\|\Phi\|_{\cc^1}\|\alpha\|$.

All this implies
$$\|\opa_b\alpha\|^2\leq\|\opa_b^\Phi\alpha\|^2+O(\|\Phi\|_{\cc^0}(\|\alpha\|_\lc^2+\|\alpha\|_{\ol\lc}^2))+O((1+\|\Phi\|_{\cc^1})^2\|\alpha\|^2)$$
and
$$\|\opa_b^*\alpha\|^2\leq\|{\opa_b^\Phi}^*\alpha\|^2+O(\|\Phi\|_{\cc^1}(\|\alpha\|_\lc^2+\|\alpha\|_{\ol\lc}^2))+O(\|\Phi\|_{\cc^1}^2\|\alpha\|^2).$$

The {\sl a priori} estimates proved in Theorem 3.1 of \cite{ShWa} gives the existence of a constant $C_0$ such that for any $\alpha\in\ci_{0,r}(\mb)$, $r=0,1$, the following estimate is satisfied
$$\|\alpha\|_\lc^2+\|\alpha\|_{\ol\lc}^2\leq C_0(\|\opa_b\alpha\|^2+\|\opa_b^*\alpha\|^2+\|\alpha\|^2).$$
Using the previous calculations we get the existence of two constants $C$ and $C'$ such that
$$\|\alpha\|_\lc^2+\|\alpha\|_{\ol\lc}^2\leq C(\|\opa_b^\Phi\alpha\|^2+\|{\opa_b^\Phi}^*\alpha\|^2+\|\alpha\|^2)+C'\|\Phi\|_{\cc^1}(\|\alpha\|_\lc^2+\|\alpha\|_{\ol\lc}^2+\|\alpha\|^2).$$
Choose $\delta<1$ sufficiently small such that $C'\delta<1$, then if $\|\Phi\|_{\cc^1}<\delta$ there exists a constant $C''$ such that
$$\|\alpha\|_\lc^2+\|\alpha\|_{\ol\lc}^2\leq C''(1+C'\|\Phi\|_{\cc^1})(\|\opa_b^\Phi\alpha\|^2+\|{\opa_b^\Phi}^*\alpha\|^2+\|\alpha\|^2).$$
This implies the theorem since under the condition Y(1) the real and the imaginary part of the vector fields $L_j$ satisfy H\"{o}rmander's finite type condition of type 2.
\end{proof}

In the spirit of the pionnier works by Kohn-Nirenberg and Folland-Kohn, following the methods used in section 3 of \cite{ShWa}, the {\sl a priori} estimates obtained in Theorem \ref{apriori} gives the existence and the regularity of solutions for the $\Box_b^\Phi$ and the $\opa_b^\Phi$ equations on compact CR manifolds. Denote by
$$\hc_{0,1}^\Phi(\mb)=\ker\Box_b^\Phi=\{\alpha\in L^2_{0,1}(\mb)~|~\opa_b^\Phi\alpha={\opa_b^\Phi}^*\alpha=0\}$$
and by $H_b^\Phi$ the projection operator from $L^2_{0,1}$ onto $\hc_{0,1}^\Phi(\mb)$. We have the following strong Hodge decomposition
\begin{thm}\label{hodge}
Let $(\mb, H_{0,1}\mb)$ be an abstract compact CR
manifold of class $\cc^\infty$ and $\widehat
H_{0,1}\mb$ be an horizontal perturbation of $H_{0,1}\mb$ defined by
a $(0,1)$-form $\Phi\in\cc_{0,1}^{l}(\mb,H_{1,0}\mb)$, $l\geq 1$. Suppose that the Levi form of
$\mb$ satisfies condition Y(1) and $\|\Phi\|_{\cc^{1}}$ is sufficiently small, then there exists a compact operator
$G_b^\Phi~:~L^2_{0,r}(\mb)\to Dom(\square_b^\Phi)$, $r=0,1$ such that

(i) For all $s\in\nb$ and $r=0,1$, $G_b^\Phi$ is continuous from $W^s_{0,r}(\mb)$ into
$W^{s+1}_{0,r}(\mb)$, more precisely there exists $\delta>0$ such that, if $\|\Phi\|_{\cc^{2+s}}<\delta$, there exists a constant $C_s$ independent of $\Phi$ such that
$$\|G_b^\Phi(\alpha)\|_{s+1}\leq C_s\|\alpha\|_s,$$
for $\alpha \in W^s_{0,r}(\mb)$.

(ii) For any $f\in L^2(\mb)$
\begin{equation*}
f={\opa_b^\Phi}^*\opa_b^\Phi G_b^\Phi f+H_b^\Phi f,
\end{equation*}
and for any $\alpha\in L^2_{0,1}(\mb)$
\begin{equation*}
\alpha=\opa_b^\Phi{\opa_b^\Phi}^*G_b^\Phi\alpha+{\opa_b^\Phi}^*\opa_b^\Phi G_b^\Phi\alpha+H_b^\Phi \alpha.
\end{equation*}

(iii) $G_b^\Phi H_b^\Phi=H_b^\Phi G_b^\Phi=0$, $G_b^\Phi\Box_b^\Phi=\Box_b^\Phi G_b^\Phi=I-H_b^\Phi$ on $Dom(\Box_b^\Phi)$
and if $G_b^\Phi$ is also defined on $L^2_{0,2}(\mb)$ (respectively
$L^2_{0,0}(\mb)$), $G_b^\Phi\opa_b^\Phi=\opa_b^\Phi G_b^\Phi$ on $Dom(\opa_b^\Phi)$ (respectiely
$G_b^\Phi{\opa_b^\Phi}^*={\opa_b^\Phi}^*G_b^\Phi$ on $Dom({\opa_b^\Phi}^*)$).
\end{thm}
From this theorem we can deduce some results on the solvability of \eqref{dbar-perturb}.
\begin{cor}\label{estimate}
Under the hypotheses of Theorem \ref{hodge}, for any $f\in L^2(\mb)$
\begin{equation*}
\opa_b^\Phi f=\opa_b^\Phi{\opa_b^\Phi}^* G_b^\Phi\opa_b^\Phi f.
\end{equation*}
Moreover there exists $\delta>0$ such that, if $\|\Phi\|_{\cc^{2+s}}<\delta$, there exists a constant $C_s$ independent of $\Phi$ such that
$$\|{\opa_b^\Phi}^* G_b^\Phi\opa_b^\Phi f\|_{s+\frac{1}{2}}\leq C_s\|\opa_b^\Phi f\|_s,$$
\end{cor}
\begin{proof}
We apply Theorem \ref{hodge} in degree $0$. If $f\in L^2(\mb)$ we get
$$f={\opa_b^\Phi}^*\opa_b^\Phi G_b^\Phi f+H_b^\Phi f,$$
which implies
$$\opa_b^\Phi f=\opa_b^\Phi{\opa_b^\Phi}^*\opa_b^\Phi G_b^\Phi f+\opa_b^\Phi H_b^\Phi f.$$
Using that $\ker\Box_b^\Phi\subset\ker\opa_b^\Phi$ and that $G_b^\Phi$ exists in bidegree $(0,1)$ and commute with $\opa_b^\Phi$ since $\mb$ satisfies condition Y(1) we obtain
$$\opa_b^\Phi f=\opa_b^\Phi{\opa_b^\Phi}^* G_b^\Phi\opa_b^\Phi f.$$
The estimate can be deduced in a classical way from the estimate (i) in Theorem \ref{hodge}, see e.g. Theorem 8.4.14 in \cite{ChSh}.
\end{proof}

\subsection{Stability theorem}

Let us go back to the setting of the beginning of Section \ref{s4}. Note that if the CR manifold $(\mb, H_{0,1}\mb)$ is $2$-concave its Levi form satisfies condition Y(1) and so we can apply the previous results on the solvability of the $\opa_b^\Phi$-equation to the equations
\begin{equation}\label{dbar-proj2}
\opa_b^\Phi g_j=\opa_b^\Phi f_j=\Phi\lrcorner\pa_b f_j\quad {\rm for~all}~ j=1,\dots,N.
\end{equation}
Assume that $\|\Phi\|_{\cc^{l+2}}$ is sufficiently small to apply Corollary \ref{estimate}, then there exists a constant $C$ independent of $\Phi$
such that if $g_j$ satisfies \eqref{dbar-proj2} then
$$\|g_j\|_{\cc^l}\leq C\|\Phi\|_{\cc^{l}}\|f_j\|_{\cc^{l+1}},$$
and taking $\Phi$ even smaller $(f_1-g_1,\dots,f_N-g_N)$  will define homogeneous CR coordinates on $(\mb,\widehat H_{0,1}\mb)$.
So we have proved
\begin{thm}\label{perturbPn}
Let $(\mb,H_{0,1}\mb)$ be an abstract compact CR
manifold of class $\cc^\infty$, which is  smoothly embeddable as a
CR manifold in $\cb P^n$. Assume that $(\mb,H_{0,1}\mb)$
is $2$-concave and let $\widehat
H_{0,1}\mb$ be an horizontal perturbation of $H_{0,1}\mb$ defined by
a $(0,1)$-form $\Phi\in\cc_{0,1}^{l+2}(\mb,H_{1,0}\mb)$, $l\geq 1$.

Then there exists a positive real number $\delta$ such that if
$\|\Phi\|_{\cc^{l+2}}<\delta$, then the CR manifold $(\mb,\widehat
H_{0,1}\mb)$ is embeddable in $\cb P^n$ as a CR submanifold of class
$\cc^l$.
\end{thm}


\providecommand{\bysame}{\leavevmode\hbox to3em{\hrulefill}\thinspace}
\providecommand{\MR}{\relax\ifhmode\unskip\space\fi MR }
\providecommand{\MRhref}[2]{%
  \href{http://www.ams.org/mathscinet-getitem?mr=#1}{#2}
}
\providecommand{\href}[2]{#2}
\begin{thebibliography}{}

\end{thebibliography}


\begin{thebibliography}{10}

\bibitem{BaLa}
M.~Y. Barkatou and C.~Laurent-Thi{\'e}baut, \emph{Estimations optimales pour
  l'op{\'e}rateur de {C}auchy-{R}iemann tangentiel}, Michigan Math. Journal
  \textbf{54} (2006), 545--586.

\bibitem{Bi}
O.~Biquard, \emph{Sur les vari\'{e}t\'{e} de {C}auchy-{R}iemann dont le forme
  de {L}evi a une valeur propre positive}, Math. Zeit. \textbf{249} (2005),
  411--425.

\bibitem{ChSh}
S.-C. Chen and M.-C. Shaw, \emph{{P}artial differential equations in several
  complex variables}, Studies in Advanced Math., vol.~19, AMS-International
  Press, 2001.

\bibitem{Gr}
J.~Gray, \emph{Some global properties of contact structures}, Ann. of Math.
  \textbf{69} (1959), 421--450.

\bibitem{Ha1}
R.S. Hamilton, \emph{Deformation of complex structure on manifolds with
  boundary i}, J. Diff. Geometry \textbf{12} (1977), 1--45.

\bibitem{Ha2}
\bysame, \emph{Deformation of complex structure on manifolds with boundary ii},
  J. Diff. Geometry \textbf{14} (1979), 409--473.

\bibitem{Ki2}
G.K. Kiremidjian, \emph{Extendible pseudocomplex structures {II}}, J. d'Analyse
  Mathématiques \textbf{30} (1976), 304--329.

\bibitem{Ki1}
\bysame, \emph{Extendible pseudocomplex structures {I}}, J. of Approximation
  Theory \textbf{19} (1977), 281--303.

\bibitem{Laanis}
C.~Laurent-Thi{\'e}baut, \emph{Local and global homotopy formulas with sharp
  anisotropic {H}{\"o}lder estimates in q-concave {C}{R} manifolds}, Preprint.

\bibitem{LaLeglobal}
C.~Laurent-Thi{\'e}baut and J.~Leiterer, \emph{Global homotopy formulas on
  q-concave {CR} manifolds for small degrees}, J. Geom. Anal. \textbf{18}
  (2008), 511--536.

\bibitem{NaPo}
M.~Nacinovich and E.~Porten, \emph{$\mathcal{C}^\infty$-hypoellipticity and
  extension of {CR} functions}, Preprint: ar{X}iv 1201.1704.

\bibitem{Poly2}
P.~L. Polyakov, \emph{Versal embeddings of compact $3$-pseudoconcave
  $cr$-submanifolds}, Math. Zeitschrift \textbf{248} (2004), 267--312.

\bibitem{ShWa}
M.-C. Shaw and L.~Wang, \emph{{H}{\"o}lder and ${L}^p$ estimates for
  $\square_b$ on {CR} manifolds of arbitrary codimension}, Math. Ann.
  \textbf{331} (2005), 297--343.

\bibitem{Ta}
N.~Tanaka, \emph{A differential geometric study on strongly pseudo-convex
  manifolds}, Kinokuniya Bookstore Co., Ltd., Kyoto, 1975.

\end{thebibliography}
\providecommand{\bysame}{\leavevmode\hbox to3em{\hrulefill}\thinspace}
\providecommand{\MR}{\relax\ifhmode\unskip\space\fi MR }
\providecommand{\MRhref}[2]{%
  \href{http://www.ams.org/mathscinet-getitem?mr=#1}{#2}
}
\providecommand{\href}[2]{#2}

\enddocument

\end